\documentclass[10pt,leqno]{amsart}
\usepackage{indentfirst, amssymb,amsmath,amsthm}
\theoremstyle{definition}

\newtheorem{lem}{Lemma}[section]

\newtheorem{rem}{Remark}[section]

\newcommand{\ol}{\overline}
\newcommand{\be}{\begin{equation}}
\newcommand{\ee}{\end{equation}}
\newcommand{\beas}{\begin{eqnarray*}}
\newcommand{\eeas}{\end{eqnarray*}}
\newcommand{\bea}{\begin{eqnarray}}
\newcommand{\eea}{\end{eqnarray}}

\numberwithin{equation}{section}
\begin{document}
\title[FUNDAMENTAL THEOREM OF ALGEBRA
- A NEVANLINNA THEORETIC PROOF]{FUNDAMENTAL THEOREM OF ALGEBRA
- A NEVANLINNA THEORETIC PROOF}
\date{}
\author[B. Chakraborty ]{Bikash Chakraborty }
\date{}
\address{Department of Mathematics, Ramakrishna Mission Vivekananda Centenary College, Rahara, West Bengal, India-700 118}
\email{bikashchakraborty.math@yahoo.com, bikashchakrabortyy@gmail.com}
\maketitle
\let\thefootnote\relax
\begin{abstract} There are several proofs of the Fundamental Theorem of Algebra, mainly using algebra, analysis and topology. In this article, we have shown that the Fundamental Theorem of Algebra can be proved using Nevanlinna's first fundamental theorem as well as second fundamental theorem also.
\end{abstract}
\section{Introduction}
One of the celebrated theorem in mathematics is Fundamental theorem of algebra, which tells that \emph{every non constant polynomial over $\mathbb{C}$ has a root in $\mathbb{C}$}.\par
Some mathematicians, like Peter Roth (Arithmetica Philosophica, 1608), Albert Girard (L'invention nouvelle en I'Alg\`{e}bre, 1629), and Re\'{n}e Descartes told some version of Fundamental Theorem of Algebra in the early $17^{\text{th}}$ century in their writings.\par
But first attempt at proving of this theorem was made in 1746 by ean Le Rond d�Alembert, unfortunately his proof was incomplete. Also attempts were made by Euler(1749), de Foncenex(1759), Lagrange(1772) and Laplace(1795) but Carl Friedrich Gauss is credited with producing the first correct proof in his doctoral dissertation of 1799, although in that proof also had some small gaps.\par
A rigorous proof was first produced by Argand in 1806 and the first text book containing the proof is Cours d'analyse de I'\'{E}cole Royale Polytechnique(1821) due to Cauchy.\par
Now a days, there are many proofs of the Fundamental Theorem of Algebra, mainly using algebra, analysis and topology (\cite{2}).\par
There are several analytical proofs using complex analysis, for examples, proof based on Liouville's Theorem, Rouche's Theorem, the Maximum Principle, Picard's Theorem, and the Cauchy Integral Theorem, Open mapping theorem etc.\par
The aim of this article is to produce another analytical proof of Fundamental Theorem of Algebra, using Nevanlinna theory.
\section{Basic notations of Nevanlinna Theory}
The theory of meromorphic functions was greatly developed by Rolf Nevanlinna (\cite{5}) during the 1920's. In both its scope and its power, his approach greatly surpasses previous results, and in his honor the field is now also known as \textbf{Nevanlinna  theory}.\par
Before going to describe the our new proof, we briefly state some definitions, notations, estimations like the Second Main Theorem of Nevanlinna (\cite{f,g}) which we used vividly throughout our journey.\par
Let $f(z)$ be a function that is meromorphic (i.e., regular except poles) and non-constant  in the complex plane. We denote
by $n(r,a; f)$ the number of $a$- points, with due count of multiplicity, of $f(z)$ in the $|z|<r$  for $a\in\mathbb{C}\cup\{\infty\}$, where an $a$-point is
counted according to its multiplicity. We put
\bea N(r,a;f)=\int_{0}^{r}\frac{n(t, a; f)-n(0, a; f)}{t}dt+n(0, a; f) \log r,\eea
where $n(0,a;f)$ denotes the multiplicity of $a$- points of $f(z)$ at the origin. Similarly we define $\ol{N}(r,a;f)$ where $a$-points of $f$ are counted without multiplicity. Next we define
\bea m(r,\infty;f)=m(r,f)=\frac{1}{2\pi}\int\limits_{0}^{2\pi}\log^{+}|f(re^{i\theta})|d\theta,\eea
where $ \log^{+}x := \max\{\log x,0\}$ for $x\geq0$.\par
Also the {\bf{Nevanlinna's characteristic function}} of $f$ is defined as $$T(r, f) = m(r, f) + N(r, f).$$
The basic estimate is the First Fundamental Theorem of Nevanlinna, for every complex number $a$, finite or infinite
\bea\label{fft}T(r, f)= m(r,a f) + N(r,a;f)+O(1)\eea
as $r\to\infty$. This result provides an upper bound to the number of roots of the equation $f(z)=a$ for all $a$.\par
But the difficult question of lower bounds of the number of roots of the equation $f(z)=a$ is answered by Second Fundamental Theorem of Nevanlinna.\par
Suppose that $f(z)$ is a non constant meromorphic function in whole complex plane. A quantity $\Delta$ is said to be $S(r,f)$ if $\frac{\Delta}{T(r,f)} \to 0 $ as $r \to \infty$ and $r \not\in E_{0}$ where $E_{0}$ is a set whose linear measure is not greater than $2$.\par
\section{Lemmas}
In this section, we state some results which we need in due course of time.
\begin{lem}\label{poly}(\cite{f})
For a non-constant polynomial function $P$
$$T(r,P)=deg(P)\log r + O(1),$$
as $r\to \infty$, where deg(P) is the degree of the polynomial $P$ and $O(1)$ is a bounded quantity $a$.
\end{lem}
\begin{lem}(\cite{g})\label{fft}(Nevanlinna's First Fundamental Theorem)
Let $f(z)$ be a non-constant meromorphic function defined in $|z|< R~(0<R \leq \infty)$ and let $a\in \mathbb{C} \cup\{\infty\}$ be any complex number. Then for $0<r<R$
$$T\left(r,\frac{1}{f-a}\right)=T(r,f)+O(1),$$
where $O(1)$ is a bounded quantity depending on $f$ and $a$ but not on  $r$.
\end{lem}
\begin{lem}(\cite{g})(page no. 23)\label{1}(Second Fundamental Theorem)
Suppose that $f(z)$ is a non-constant meromorphic function in the complex plane and $a_{1},a_{2},...,a_{q}$ are $q~(\geq3)$ distinct values in $\mathbb{C}\cup\{\infty\}$ . Then
$$(q-2)T(r,f)< \sum\limits_{j=1}^{q}\ol{N}(r,a_{j};f+S(r,f)$$
where $S(r,f)$ is a quantity such that $\frac{S(r,f)}{T(r,f)} \to 0$ as $r \to+\infty$ out side of a set $E$ in $(0,\infty)$ with finite linear measure.
\end{lem}
\begin{lem}\label{Claim 1}
A polynomial of the form $$Q(z)=b_{0}z^{m}+b_{1}z^{m-1}+b_{2}z^{m-2}+...+b_{m-l}z^{l}+b_{m}$$
has atleast one zero, where $m$ and $l$ are positive integers satisfying $m>2$ , $m>l\geq 2$ and  $b_{i}\in \mathbb{C}$ for $i=1,2,3,...,m-2,m$ with $b_{0}\not=0$.
\end{lem}
\begin{proof}
On contrary, we assume that the polynomial $Q(z)$ has no zero, then by definition, $\ol{N}(r,0;Q)=0$.\par
Next we define $F(z):=z^{l}R(z)$ where $R(z)=b_{0}z^{m-l}+b_{1}z^{m-l-1}+b_{2}z^{m-l-2}+...+b_{m-2}.$ Then
\bea Q(z)=F(z)+b_{m},\eea
Thus applying the Second Fundamental Theorem and Lemma \ref{poly}, we have
\beas T(r,F) &<& \ol{N}(r,\infty;F)+\ol{N}(r,0;F)+\ol{N}(r,-b_{m};F)+S(r,F)\\
&<& \ol{N}(r,0;z^{l})+\ol{N}(r,0;R(z))+\ol{N}(r,0,Q(z))+S(r,F)\\
&\leq& \frac{1}{l}N(r,0;z^{l})+N(r,0;R(z))+S(r,F)\\
&=& \log r+(m-l)\log r +S(r,F)\\
&=& \left(\frac{m-l+1}{m}\right)T(r,F)+S(r,F),
\eeas
which is absurd, hence our assumption is wrong. Thus $Q(z)$ has atleast one zero. Hence the proof.
\end{proof}
\section{Proof of the Fundamental Theorem of Algebra}
Assume {$\textbf{S(n)}$} denotes the following statement :\par \emph{Any $n$-degree polynomial over $\mathbb{C}$ has atleast one zero for any $n\in \mathbb{N}$.}\medbreak
Thus it is sufficient to prove that the statement $S(n)$ is true for all $n \in \mathbb{N}$, and we have to prove that the statement $S(n)$ is true by using mathematical induction on $n$.\par
It is obvious that the statements $S(1)$ and $S(2)$ are true, and assume that $S(k)$ is true for any natural number $k\geq3$. Now we have to show that $S(k+1)$ is true .\par
For this, we consider any polynomial of degree $k+1$ over $\mathbb{C}$ as $$P(z)=a_{0}z^{k+1}+a_{1}z^{k}+a_{2}z^{k-1}+...+a_{k}z+a_{k+1}$$  with $a_{0}\not=0$.\par
Since the statement $S(k)$ is true, so that we can find a suitable complex number $h$ such that the coefficient of $z$ in $P(z+h)$ is zero. So by Lemma \ref{Claim 1},
$P(z+h)$ has atleast one zero. Consequently, $P(z)$ has atleast one zero. Hence the statement $S(k+1)$ is true.\par
Thus by mathematical induction,  any $n$-degree non-constant polynomial over $\mathbb{C}$ has atleast one zero for any $n\in \mathbb{N}$. Hence the proof.
\begin{rem} If $P(z)$ be a non-constant polynomial having no zero in  $\mathbb{C}$, then by First Fundamental Theorem, one can deduce that \bea\label{s1}  n\log r=O(1),\eea
for sufficiently large $r$ where $z:=re^{i\theta}$ and  $O(1)$ is a bounded term dependent on $P(z)$. But equation (\ref{s1})  is absurd. Thus one can also deduce Fundamental Theorem  Algebra from  First Fundamental Theorem.
\end{rem}

\end{document}